\documentclass{amsart}
\usepackage[T1]{fontenc}
\usepackage[margin=1.5 in]{geometry}
\usepackage{dirtytalk}
\usepackage[dvipsnames]{xcolor}
\usepackage{float}

\usepackage{mathtools}
\usepackage{blkarray, bigstrut}
\usepackage{nicematrix}

\usepackage{amsmath}
\usepackage{amssymb}
\usepackage{amsthm}
\usepackage{tikz-cd}
\usepackage{tikz}
\usetikzlibrary{knots, snakes}
\usetikzlibrary{arrows.meta}
\usepgflibrary[arrows.meta]
\usetikzlibrary{angles, quotes, decorations.markings, decorations.pathmorphing, decorations.pathreplacing, decorations.shapes}
\usepackage{graphicx}
\usepackage{mathtools}
\usepackage{extarrows}
\usepackage{booktabs}
\usepackage{pythonhighlight}
\usepackage{hyperref}
\usepackage[all]{xy}
\usepackage{stmaryrd}
\usepackage{url}
\usepackage{accents}
\usepackage[
backend=bibtex,
style=numeric,
sorting=nyt, 
maxbibnames=99]{biblatex}

\usepackage{amsmath,amsthm,amssymb,latexsym,amscd}
\usepackage{mathrsfs}

\usepackage{hyperref}
\hypersetup{pdftitle={Buchanan A condition on the Khovanov homology of three families of positive links},pdfauthor={Lizzie Buchanan}}
\hypersetup{colorlinks=true,linkcolor=blue,anchorcolor=blue,citecolor=blue}
\numberwithin{equation}{section}
\theoremstyle{plain}
 
\newtheorem{lemma}[equation]{Lemma} 
\newtheorem*{lemma*}{Lemma}

\newtheorem{proposition}[equation]{Proposition}
\newtheorem*{proposition*}{Proposition}

\newtheorem{theorem}[equation]{Theorem}
\newtheorem*{theorem*}{Theorem}

\newtheorem{corollary}[equation]{Corollary}
\newtheorem*{corollary*}{Corollary}

\newtheorem*{conjecture*}{Conjecture}
\newtheorem{question}[equation]{Question}

\def\ol{\overline}

\definecolor{purp}{RGB}{148,30,238}
\definecolor{pinky}{RGB}{255, 174, 252}

\definecolor{purp1}{RGB}{120,30,238}
\definecolor{pinky1}{RGB}{243, 162, 252}
\definecolor{pinky2}{RGB}{252, 165, 249}
\definecolor{pinky3}{RGB}{247, 174, 255}

\definecolor{purp2}{RGB}{189, 144, 249}
\definecolor{purp3}{RGB}{203, 172, 243}

\theoremstyle{remark}

\theoremstyle{definition}

\newtheorem{definition}[equation]{Definition}

\DeclareMathOperator{\coeff}{coeff}
\DeclareMathOperator{\lead}{lead}

\DeclareMathOperator{\rank}{rank}

\def\Z{\mathbb Z}

\title{A condition on the Khovanov homology of three families of positive links}
\author{Lizzie Buchanan}

\begin{document}

\maketitle

\begin{abstract}
    In previous work, we developed diagram-independent upper bounds on the maximum degree of the Jones polynomial of three families of positive links. These families are characterized by the second coefficient of the Jones polynomial. In this paper, we extend those results and construct diagram-independent upper bounds on the maximum non-vanishing quantum degree of the Khovanov homology of three families of positive links. This can be used as a positivity obstruction.
\end{abstract}

\section{Introduction}

Khovanov homology is a link invariant introduced in \cite{KhovanovCat} by Khovanov that categorifies the Jones polynomial. While the question of whether or not the Jones polynomial detects the unknot famously remains unanswered, Kronheimer and Mrowka showed in \cite{KronheimerMrowka11} that Khovanov homology does detect the unknot. We direct the reader to Bar-Natan's \cite{BarNatan02}, Khovanov's \cite{KhovanovCat} and \cite{Khovanov2003}, and Viro's \cite{OlegViro2004} for full description of its construction.
For this paper, the relevant information is that the bigraded Euler characteristic of Khovanov homology is, up to normalization and change of variables, the Jones polynomial $V_L(t)$ defined by Jones in \cite{Jones:1985dw}. 
As such, is natural to try to lift statements about the Jones polynomial to statements about Khovanov homology. 

In \cite{Buchanan_2022}, \cite{Buchanan2023}, and \cite{buchanan2025conditionjonespolynomialfamily}, we proved that positive links with second Jones coefficient equal to $0, \pm 1,$ or $\pm 2$ satisfy inequalities relating the maximum and minimum degrees of the Jones polynomial and the leading coefficient of the Conway polynomial. 
(With this terminology the Jones polynomial of the positive trefoil knot, $t + t^3 - t^4$, has second Jones coefficient equal to $0$.)

\begin{theorem}[\cite{Buchanan_2022, Buchanan2023, buchanan2025conditionjonespolynomialfamily}] \label{results for Jones}
    Let $L$ be a positive link with Jones polynomial $V_L$. Let $p_1(L)$ be the absolute value of the second coefficient of $V_L$. If $p_1(L) = 0,$ $1$, or $2$, then

    $$ \max \deg V_L \leq \begin{cases}
        4\min\deg V_L + \frac{n-1}{2} & \text{ if } p_1(L) = 0,\\
        4\min\deg V_L + \frac{n-1}{2} + 2\lead\coeff\nabla_L - 2 & \text{ if } p_1(L) = 1,\\
        4\min \deg V_L + \frac{n-1}{2} + \lead\coeff\nabla_L & \text{ if } p_1(L) = 2,
    \end{cases}
    $$
\noindent
    where $n$ is the number of link components and $\nabla_L$ is the Conway polynomial of $L$. 

\end{theorem}

Theorem \ref{results for Jones} can serve as a positivity obstruction. In \cite{Buchanan2023} we constructed infinite families of almost-positive knot diagrams with second Jones coefficient equal to $0, \pm 1, $ or $\pm 2$. We showed that these knots are not positive precisely because they fail the test of Theorem \ref{results for Jones}.

In this paper we generalize those results about the Jones polynomial to statements about Khovanov homology, and provide an example of a knot which demonstrates that, at least for links with second Jones coefficient equal to $0$, our new Theorem \ref{generalizes to khovanov, cases by 2nd coeff} is a stronger positivity obstruction than Theorem \ref{results for Jones}.\\

\noindent
\textbf{Theorem \ref{generalizes to khovanov, cases by 2nd coeff}}.
\textit{    Let $L$ be a positive link and let $p_1(L)$ be the absolute value of the second coefficient of its Jones polynomial. If $p_1(L) = 0,$ $1$, or $2$, then}

    $$ \ol{j}(L) \leq \begin{cases}
        4\underline{j}(L) + n + 4 & \text{ if } p_1(L) = 0,\\
        4\underline{j}(L) + n + 4\lead\coeff\nabla_L & \text{ if } p_1(L) = 1,\\
        4\underline{j}(L) + n + 4 + 2\lead\coeff\nabla_L & \text{ if } p_1(L) = 2,
    \end{cases}
    $$\\
\noindent
     \textit{where the second coefficient is the coefficient of the $t^{\min\deg V_L + 1}$ term, $n$ is the number of link components, $\ol{j}(L):= \max\{j | Kh^{*,j}(L)\neq 0\}$,  $\underline{j}(L):= \min\{j | Kh^{*,j}(L) \neq 0\}$, and $\nabla_L$ is the Conway polynomial of $L$. }\\

\vspace{-4pt}

The value $p_1(L)$, called the \textit{cyclomatic number} in \cite{RadSDS22} and \cite{kegel2023khovanovhomologypositivelinks} of a positive link $L$, has graph-theoretic, diagrammatic significance. This significance is not directly needed in this paper, so we omit its discussion for the sake of brevity. The interested reader can explore this topic further in the following:
Futer, Kalfagianni, and Purcell's \cite{futer2012guts}; Futer's \cite{Futer_2013}; Kegel, Manikandan, Mousseau, and Silvero's \cite{kegel2023khovanovhomologypositivelinks}; Przytycki and Silvero's \cite{PS2020};   Sazdanovi\'{c} and Scofield's \cite{RadSDS22}; Stoimenow's \cite{St05}; and our \cite{Buchanan_2022}, \cite{Buchanan2023}, and \cite{buchanan2025conditionjonespolynomialfamily}.

In \cite{RadSDS22}, Sazdanovi\'{c} and Scofield found that for any positive link $L$, $Kh^{1,2-\chi}(L) = \Z^{p_1(L)}$. And in \cite{kegel2023khovanovhomologypositivelinks}, Kegel, Manikandan, Mousseau, and Silvero found that all other homology groups in homological grading 1 vanish, so that for any positive link $L$ we have $Kh^1(L) = \Z^{p_1(L)}$. This gives a nice way to restate Theorem \ref{generalizes to khovanov, cases by 2nd coeff} solely in terms of Khovanov homology and the Conway polynomial, without explicit reference to the Jones polynomial at all.

\begin{theorem}\label{generalizes to khovanov, cases by hom 1}
    Let $L$ be a positive link whose Khovanov homology group in homological grading $1$ is $Kh^{1}(L) \simeq \Z^{p_1(L)}$ where $p_1(L) = 0, 1,$ or $2$. Then 

    $$ \ol{j}(L) \leq \begin{cases}
        4\underline{j}(L) + n + 4 & \text{ if } p(L) = 0,\\
        4\underline{j}(L) + n + 4\lead\coeff\nabla_L & \text{ if } p(L) = 1,\\
        4\underline{j}(L) + n + 4 + 2\lead\coeff\nabla_L & \text{ if } p(L) = 2,
    \end{cases}
    $$
\noindent
    where $\ol{j}(L):= \max\{j | Kh^{*,j}(L)\neq 0\}$,  $\underline{j}(L):= \min\{j | Kh^{*,j}(L) \neq 0\}$, $n$ is the number of link components, and $\nabla_L$ is the Conway polynomial of $L$. 

\end{theorem}

This paper is structured as follows: In Section \ref{section khovanov homology and jones} we remind the reader of the relationship between the Jones polynomial and Khovanov homology. In Section \ref{section extreme quantum grading} we investigate  several definitions around the quantum grading of the Khovanov homology and we prove Theorem \ref{generalizes to khovanov, cases by 2nd coeff}.
Finally, in Section \ref{section further exploration} we demonstrate that Theorem \ref{generalizes to khovanov, cases by 2nd coeff} is strictly stronger than Theorem \ref{results for Jones} as a positivity obstruction, and we present a few questions for future research.

\section*{Acknowledgements}
The author is partially supported by NSF Grant DMS-2038103 at the University of Iowa. She would like to thank Keiko Kawamuro for guidance on this project, and Marithania Silvero for comments on an early draft of this paper. 

\section{Khovanov homology and Jones polynomial}\label{section khovanov homology and jones}

Computing the Khovanov homology of a link $L$ involves setting up $\Z$-graded chain complexes associated to a diagram $D$ of $L$, and then computing the bigraded homology groups $Kh^{i,j}(D)$ of those chain complexes. The groups $Kh^{i,j}(D)$ are actually invariants of the link $L$, so can be denoted $Kh^{i,j}(L)$ \cite{KhovanovCat}. Such a group $Kh^{i,j}(L)$ is called a \textit{Khovanov homology group with homological grading $i$ and quantum grading $j$}. 

Khovanov homology categorifies the Jones polynomial in the sense that the graded Euler characteristic of the Khovanov homology is $J_L(q)$, a version of the Jones polynomial which can be normalized to obtain the original Jones polynomial $V_L(t)$,
given in \cite{Jones:1985dw} and \cite{Kauffman}, with which the Jones polynomial of the unknot is $1$. In some texts $J_L(q)$ is called \say{the Jones polynomial} and is used instead of $V_L(t)$. The relationship between Khovanov homology groups, $J_L(q)$ (which we call \say{an unnormalized Jones polynomial}), and $V_L(t)$ is shown in the following equation.

\begin{equation}\label{jones poly conversion eqn v4}
    \underset
        {\underset
            {\text{an unnormalized Jones polynomial}}
            {=\text{graded Euler characteristic, }}
        }
        {\underbrace{\sum_{i,j}(-1)^i \rank Kh^{i,j}(L)q^j}} = J_L(q) = (q + q^{-1}) V_L(t)\Big|_{t^{1/2} \hspace{1pt}=\hspace{1pt} -q} 
\end{equation}

As Jones notes in his paper, $V_L$ is a Laurent polynomial in $\Z[t, t^{-1}]$ for links with an odd number of components, and is $t^{1/2}$ times a Laurent polynomial in $\Z[t, t^{-1}]$ for links with an even number of components \cite{Jones:1985dw}.
Hence for knots in particular, the step of replacing $t^{1/2}$ in $V_L(t)$ with $-q$ is equivalent to simply replacing $t^{1/2}$ with $q$. For that reason, some texts which deal exclusively with knots omit that minus sign in the setup of equation \ref{jones poly conversion eqn v4}. 

The key data of the Khovanov homology is, of course, the homology groups, which we can organize in a chart as in the following examples. 

\begin{figure}[h]
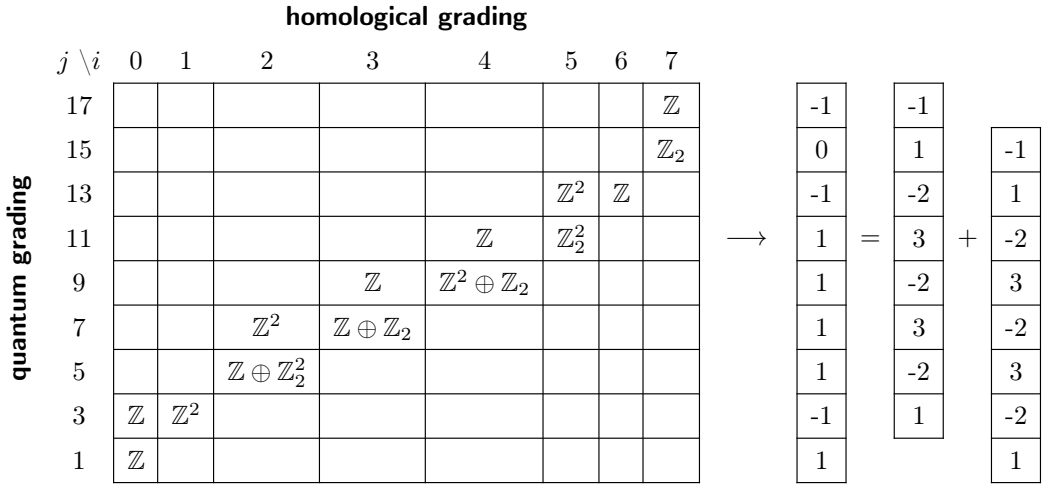

\centering
\renewcommand{\arraystretch}{1.4}
\begin{NiceTabular}{*{20}{c}}
&   & \Block{1-*}{\bfseries \sffamily homological grading} \\
& $j$ \textbackslash $i$  & 0 & 1 & 2 & 3 & 4 & 5 & 6 & 7  \\
\Block{*-1}{\rotatebox{90}{\bfseries \sffamily quantum grading}} 
& 17 &  \Block[hvlines]{*-*}{}    
            &     &       &     &          &   & &  $\Z$    \\ 
& 15 &      &     &       &        &       &    &  & $\Z_2$    \\ 
& 13 &      &     &       & & & $\Z^2$   & $\Z$  &          \\ 
& 11 &      &     &       & &  $\Z$ & $\Z_2^2$  &        \\ 
& 9 & &&& $\Z$ & $\Z^2 \oplus \Z_2$ \\
& 7 & && $\Z^2$ & $\Z \oplus \Z_2$ \\
& 5 & && $\Z \oplus \Z_2^2$ \\
& 3 & $\Z$ & $\Z^2$ \\
& 1 & $\Z$ &  \\
\end{NiceTabular}
\hspace{3pt} $\longrightarrow$\hspace{-2pt}
%
\begin{NiceTabular}{*{20}{c}}
&   \\
&   \\ 
&  \Block[hvlines]{*-*}{}    
         -1    \\ 
& 0  \\ 
& -1  \\ 
& 1  \\ 
& 1  \\
& 1  \\ 
& 1  \\
& -1  \\ 
& 1  \\
\end{NiceTabular}
\hspace{2pt}$=$\hspace{-8pt}
\begin{NiceTabular}{*{20}{c}}[corners=SE]
&   \\
&   \\ 
&  \Block[hvlines]{*-*}{}    
         -1         \\ 
& 1 \\ 
& -2 \\ 
& 3 \\ 
& -2  \\ 
& 3  \\
& -2  \\ 
& 1  \\
& \\
\end{NiceTabular}
\hspace{2pt}$+$\hspace{-8pt}
\begin{NiceTabular}{*{20}{c}}[corners=NW]
&   \\
&   \\ 
&  \Block[hvlines]{*-*}{}    
            \\ 
& -1   \\ 
& 1 \\ 
& -2 \\ 
& 3 \\ 
& -2  \\ 
& 3  \\
& -2  \\ 
& 1  \\
\end{NiceTabular}
\caption{Khovanov homology of $7_4$ (left), coefficient vector of unnormalized Jones polynomial $J_{7_4}(q)$ (middle), conversion to coefficient vector of Jones polynomial $V_{7_4}(t)$(right). Data from KnotAtlas \cite{KnotAtlas7-4}.}
\label{fig: KH 7_4}
\end{figure}

In Figure \ref{fig: KH 7_4} we see the Khovanov homology chart of positive knot $7_4$, which we can use to calculate the coefficient vector of the unnormalized Jones polynomial $J_L(q)$ and then the coefficient vector of the Jones polynomial $V_L(t)$. The alternating sum of the ranks of the homology groups in row $j$ gives the coefficient of the $q^{j}$ term of the unnormalized Jones polynomial $J_L(q)$. Divide $J_L(q)$ by $q+q^{-1}$ and replace $q$ with $-t^{1/2}$ and we get the normalized Jones polynomial $V_L(t)$, whose coefficient vector appears on the far right of Figure \ref{fig: KH 7_4}. 

In this example we have: 
$$\text{Unnormalized Jones polynomial: } J_{7_4}(q) = q - q^3 + q^5 + q^7 + q^9 +q^{11} - q^{13} - q^{17}$$
$$\text{Jones polynomial: } V_{7_4}(t) = t-2t^2 + 3t^3 - 2t^4 + 3t^5 - 2t^6 + t^7 -t^8$$

As we can see, the largest possible value of $j$ for which row $j$ of the Khovanov chart is non-empty will tell us the largest possible value of the maximum degree of the Jones polynomial. But the reverse is not true in general. For example, in Figure \ref{fig:KH 11n57} we see that knowing the maximum degree of $V_{11_n57}(t)$ alone does not allow us to accurately predict the largest possible value of $j$ for which row $j$ of the Khovanov chart is non-empty.

\begin{figure}[h]
    \centering
    \includegraphics[width=1\linewidth]{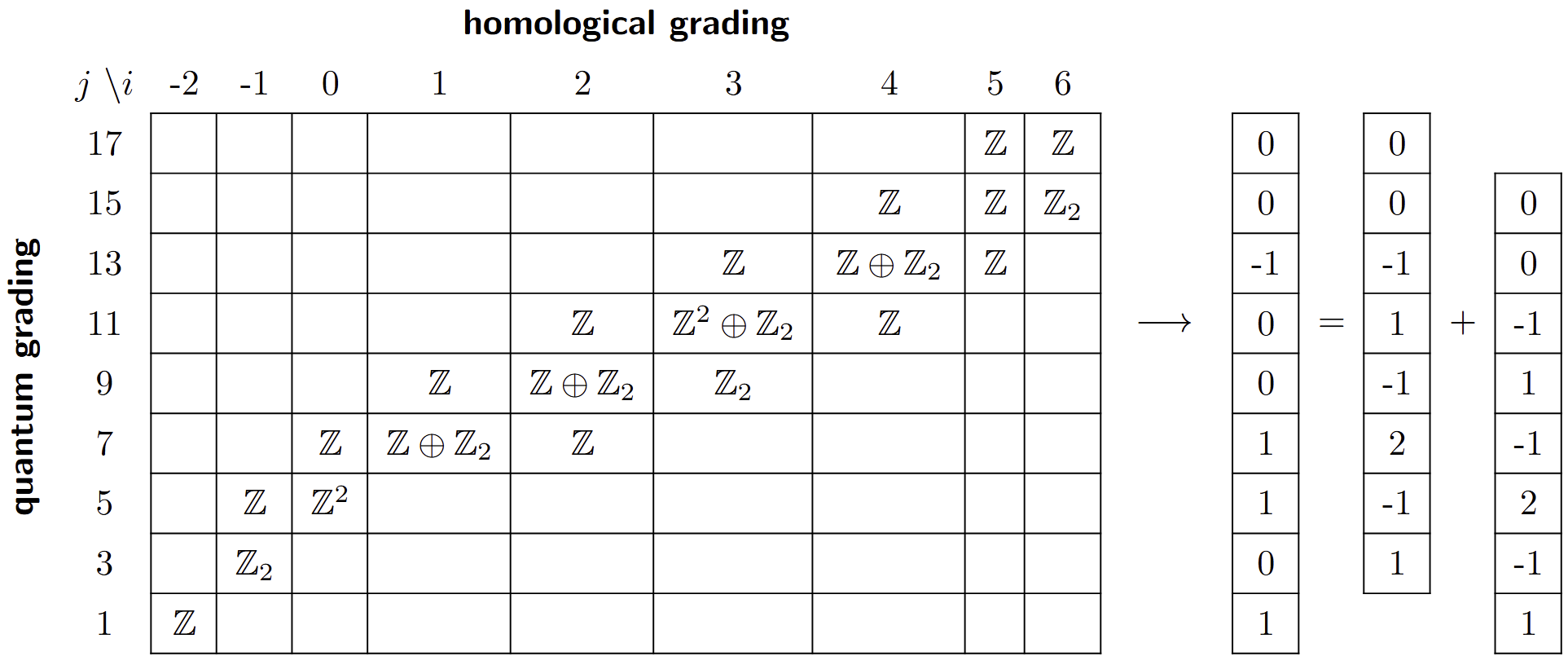}
    \caption{Khovanov homology of $11_n57$ (left), conversion to coefficient vector of its Jones polynomial (right). Data from KnotAtlas \cite{KnotAtlas11n57}.}
    \label{fig:KH 11n57}
\end{figure}

\vspace{5pt}

\section{Extreme Quantum Grading}\label{section extreme quantum grading}

The following definitions about extreme quantum gradings come largely from \cite{GMMS2018}.

\begin{definition}
    We call $$\underline{j}(L) := \min \{j \, | \, Kh^{*,j}(L) \neq 0\}$$ the \textbf{lower extreme quantum grading} of the Khovanov homology of link $L$. We can similarly define $$\ol{j}(L) := \max \{j \, | \, Kh^{*,j}(L) \neq 0\}$$ to be the \textbf{upper extreme quantum grading}. 
\end{definition}

In brief, Khovanov homology and the Jones polynomial are each calculated from a link diagram by associating a mathematical object to each \textit{state}, and then  combining all the contributions from all of the states.

\begin{definition}
At every crossing in a link diagram $D$, we can perform either an \textbf{$\textit{A}$-smoothing} or a \textbf{$\textit{B}$-smoothing} (see Figure \ref{A-smooth and B-smooth}). After smoothing all crossings, we have an arrangement of circles. 
\begin{figure}[h]
	\centering
	\begin{tikzpicture}[every path/.style={thick}, every
		node/.style={transform shape, knot crossing, inner sep=2.75pt}, scale=0.6]
		\node (tl) at (-1, 1) {};
		\node (tr) at (1, 1) {};
		\node (bl) at (-1, -1) {};
		\node (br) at (1, -1) {};
		\node (c) at (0,0) {};
		
		\draw (bl) -- (tr);
		\draw (br) -- (c);
		\draw (c) -- (tl);
		
		\begin{scope}[xshift=4cm]
			\node (tl) at (-1, 1) {};
			\node (tr) at (1, 1) {};
			\node (bl) at (-1, -1) {};
			\node (br) at (1, -1) {};

			\draw (bl) .. controls (bl.8 north east) and (tl.8 south east) .. (tl);
			\draw (br) .. controls (br.8 north west) and (tr.8 south west) .. (tr);
		\end{scope}

		\begin{scope}[xshift=8cm]
			\node (tl) at (-1, 1) {};
			\node (tr) at (1, 1) {};
			\node (bl) at (-1, -1) {};
			\node (br) at (1, -1) {};

			\draw (bl) .. controls (bl.8 north east) and (br.8 north west) .. (br);
			\draw (tl) .. controls (tl.8 south east) and (tr.8 south west) .. (tr);
		\end{scope}
		
	\end{tikzpicture}
    \vspace{-5pt}
	\caption{ A crossing (left), its $A$-smoothing (middle), and $B$-smoothing (right)}
	\label{A-smooth and B-smooth}
\end{figure}
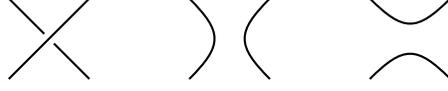 
If the diagram has $c(D)$ crossings, then we have $2^{c(D)}$ total possible arrangements of circles. Each such arrangement is called a \textbf{state}. 
\end{definition}

\begin{definition}
    The \textbf{$\textit{A}$-state} of a diagram $D$ is the state obtained by giving every crossing in the diagram an $A$-smoothing. The \textbf{$\textit{B}$-state} of a diagram $D$ is the state obtained by giving every crossing in the diagram a $B$-smoothing. We use $|s_A(D)|$ to denote the number of circles in the $A$-state and $|s_B(D)|$ to denote the number of circles in the $B$-state.
\end{definition}

\begin{definition}
	A (oriented) link diagram is \textbf{positive} if every crossing in that diagram is positive: \begin{tikzpicture}[every path/.style={thick}, every
	node/.style={transform shape, knot crossing, inner sep=5pt}, scale=0.25]
	\node (tl) at (-1, 1) {};
	\node (tr) at (1, 1) {};
	\node (bl) at (-1, -1) {};
	\node (br) at (1, -1) {};
	\node (c) at (0,0) {};
	
	\draw [-{Stealth[length=4pt,width=4pt]}] (bl) -- (tr);
	\draw (br) -- (c);
	\draw [-{Stealth[length=4pt,width=4pt]}] (c) -- (tl);
\end{tikzpicture}.
A link is \textbf{positive} if it has a positive diagram.  
\end{definition}

In \cite{GMMS2018}, Gonz\'alez-Meneses, Manch\'on and Silvero
they found that for any given link diagram $D$, the \textit{potential} extreme Khovanov gradings come from the contributions of the all-$A$ state and the all-$B$ state.

\begin{definition}
    For any given link diagram $D$ with $p$ positive crossings and $q$ negative crossings, the \textbf{potential extreme Khovanov gradings} are 
    $$j_{\min}(D) :=  c(D)-3q(D) - |s_A(D)|$$ and $$j_{\max}(D) :=  -c(D) + 3p(D) + |s_B(D)|.$$ 
\end{definition}

These definitions are justified by the following: 

\begin{proposition}[{\cite[Corollary~2]{GMMS2018}}]\label{Khovanov bounds GMS}
For any diagram $D$ of any link $L$,
$$j_{\min}(D) \leq \underline{j}(L) \hspace{10pt}\text{ and } \hspace{10pt} \ol{j}(L) \leq j_{\max}(D).$$
\end{proposition}

Specializing to a positive diagram $D$, we know from \cite{Khovanov2003} that the left-hand inequality above is actually an equality, and so we immediately have: 

\begin{corollary}\label{Khovanov bounds pos GMS}
    For any positive diagram $D$ of a link $L$, we have 

    \begin{equation*} \label{Khovanov min PS}
\underline{j}(L) = j_{\min}(D) = c(D) - |s_A(D)|
\end{equation*} 
and 
\begin{equation*} \label{Khovanov max PS}
\ol{j}(L) \leq j_{\max}(D)  = 2c(D)+|s_B(D)|.
\end{equation*}

\end{corollary}

We have similar statements for the Jones polynomial. In particular, for a positive diagram, it follows from a more general result proved by Lickorish, 

\begin{lemma}[\cite{Lickorish}] For a positive diagram $D$ of a link $L$, 
$$\min\deg V_L = \frac{c(D) - |s_A(D)| + 1}{2}$$ and
$$\max\deg V_L \leq \frac{2c(D) + |s_B(D)| - 1}{2}.$$ 
\end{lemma}

In \cite{Buchanan_2022}, \cite{Buchanan2023}, and \cite{buchanan2025conditionjonespolynomialfamily} we developed Theorem \ref{results for Jones}, diagram-independent bounds on the maximum degree of the Jones polynomial of three families of positive links. The key argument proves that if $D$ is a reduced positive diagram of a link $L$ with $p(L)=0, 1, $ or $2$,  

\begin{align}
    \max\deg V_L & \leq \frac{2c(D) + |s_B(D)| - 1}{2} \nonumber \\
    & \leq 4\Big( \frac{c(D) - |s_A(D)| + 1}{2} \Big) + \frac{n-1}{2} + \gamma(L)
    \label{bound with gamma}
\end{align} 
where 
$$\gamma(L) = \begin{cases}
    0 & \text{ if } p_1(L)=0,\\
    2\lead\coeff\nabla_L - 2 &\text{ if } p_1(L)=1,\\
    \lead\coeff\nabla_L &\text{ if } p_1(L)=2.
\end{cases}$$

That is, while Theorem \ref{results for Jones} is stated as a bound on the maximum degree of the Jones polynomial, it is really a bound on the \textit{potential} maximum degree of the Jones polynomial. Hence the key step in proving Theorem \ref{results for Jones} is not bounding $\max \deg V_L$ directly, but finding an upper bound for $|s_B(D)|$, since the potential maximum degree of the Jones polynomial comes from the contribution of the all-$B$ state to the Jones polynomial. 

Combining the inequalities in \ref{bound with gamma} and Corollary \ref{Khovanov max PS} allows us to generalize the statement about bounds on the Jones polynomial in Theorem \ref{results for Jones} to statements about the Khovanov homology. 

\begin{theorem}\label{generalizes to khovanov, cases by 2nd coeff}
  Let $L$ be a positive link and let $p_1(L)$ be the absolute value of the second coefficient of its Jones polynomial. If $p_1(L) = 0,$ $1$, or $2$, then

    $$ \ol{j}(L) \leq \begin{cases}
        4\underline{j}(L) + n + 4 & \text{ if } p_1(L) = 0,\\
        4\underline{j}(L) + n + 4\lead\coeff\nabla_L & \text{ if } p_1(L) = 1,\\
        4\underline{j}(L) + n + 4 + 2\lead\coeff\nabla_L & \text{ if } p_1(L) = 2,
    \end{cases}
    $$
      where the second coefficient is the coefficient of the $t^{\min\deg V_L + 1}$ term, $n$ is the number of link components, $\ol{j}(L):= \max\{j | Kh^{*,j}(L)\neq 0\}$,  $\underline{j}(L):= \min\{j | Kh^{*,j}(L) \neq 0\}$, and $\nabla_L$ is the Conway polynomial of $L$. 
\end{theorem}

\begin{proof}
    Let $L$ as above. Let $D$ be a reduced positive diagram of $L$. Then 

    \vspace{-3pt}

\begin{align*}
    \ol{j}(L)  \leq j_{\max}(D)  &= 2c(D)+|s_B(D)| &\text{ by \cite{GMMS2018} via \ref{Khovanov max PS}}\\
    &= 2 \Big( \frac{2c(D)+|s_B(D)|-1}{2} \Big) +1 \\
    & \leq 2 \Big( 4 \Big( \frac{c(D)-|s_A(D)|+1}{2} \Big) + \frac{n-1}{2} + \gamma(L) \Big) +1 &\text{ by \cite{Buchanan_2022, Buchanan2023, buchanan2025conditionjonespolynomialfamily} via \ref{bound with gamma}}\\
    & = 4 \Big( c(D)-|s_A(D)| + 1\Big) + n-1 + 2\gamma(L) +1 \\
    & = 4 \underline{j}(L) + n + 4 + 2\gamma(L) &\text{ by \cite{GMMS2018} via \ref{Khovanov min PS}}
\end{align*}

where 
$$\gamma(L) = \begin{cases}
    0 & \text{ if } p_1(L)=0\\
    2\lead\coeff\nabla_L - 2 &\text{ if } p_1(L)=1\\
    \lead\coeff\nabla_L &\text{ if } p_1(L)=2.
\end{cases}$$

The result follows. 

\end{proof}

\section{Further Exploration}\label{section further exploration}

In this section we consider some consequences of these results, and pose some questions for further avenues of research. Since 
$$\text{positive links} \subsetneq \text{strongly quasipositive links} \subsetneq \text{quasipositive links}, \text{\cite{Rudolph98}}$$
it is natural to ask if we can generalize any of our results to quasipositive or strongly quasipositive or links. 

\begin{corollary}
    Theorems \ref{results for Jones} and \ref{generalizes to khovanov, cases by 2nd coeff} cannot be extended to strongly quasipositive links.
\end{corollary}

\begin{proof}

In \cite{Buchanan_2022}, \cite{Buchanan2023}, and \cite{buchanan2025conditionjonespolynomialfamily} we found examples of almost-positive diagrams with second Jones coefficient equal to $0$, $\pm 1$, or $\pm 2$ whose Jones polynomials do not satisfy the conclusion of Theorem \ref{results for Jones}. Hence their Khovanov homologies cannot satisfy the conclusion of Theorem \ref{generalizes to khovanov, cases by 2nd coeff} either. Since Feller, Lewark, and Lobb proved in \cite{Feller_2022} that all links with almost-positive diagrams are strongly quasipositive, Theorems \ref{results for Jones} and \ref{generalizes to khovanov, cases by 2nd coeff} cannot be extended to strongly quasipositive links.
    
\end{proof}

But can we refine this any further?

\begin{question}
    In \cite{kegel2023khovanovhomologypositivelinks}, Kegel, Manikandan, Mousseau, and Silvero found that for any positive link $L$ we have $Kh^1(L) = \Z^{p_1(L)}$. Does this hold for strongly quasipositive links?
\end{question}

If it does, then Theorem \ref{generalizes to khovanov, cases by hom 1} could also hold for strongly quasipositive links. 

\begin{question}
    Does Theorem \ref{generalizes to khovanov, cases by hom 1} hold for strongly quasipositive links?
\end{question}

\begin{corollary}
    For knots with second Jones coefficient equal to $0$, the Khovanov test of Theorem \ref{generalizes to khovanov, cases by 2nd coeff} is strictly stronger than the Jones test of Theorem \ref{results for Jones}.
\end{corollary}

\begin{proof} All of the following polynomial data comes from the KnotInfo website \cite{knotinfo}. We consider the knot $12_n 749$. Its Jones polynomial is $t^3 + t^5 - t^6 + t^7 - t^8 + t^9 - t^{10}$, so Theorem \ref{results for Jones} cannot detect that this knot is not positive. In fact, since $12_n 749$ has the same HOMFLY-PT polynomial as positive knot $7_1$, any condition on the HOMFLY-PT, Conway, or Jones polynomials of a positive knot will fail to detect that $12_n749$ is not positive.

But next we look at the Khovanov homology of $12_n749$. Its Khovanov unreduced integral polynomial (polynomial in which each term $a t^i q^j$ corresponds to a copy of $\Z^a$ in row $j$ column $i$ of the Khovanov homology chart, and each $a t^i q^j T^2$ corresponds to a copy of $\Z_2^a$) is 
$$(1 + t)q^{3} + q^{5} + (2 t^{2}  + t^{3})q^{7} +  t^{4}q^9 + (t^{3} + 2 t^{4} + t^{5})q^{11} + (t^{5} + t^{6}) q^{13} + (t^{5} + t^{6}) q^{15}  + (t^{7}  + t^{8}) q^{17} + t^{9} q^{21}$$
$$+ t^{2} q^{5} T^{2} + t^{3} q^{9} T^{2} + t^{4} q^{9} T^{2} + t^{6} q^{13} T^{2} + t^{7} q^{15} T^{2} + t^{9} q^{19} T^{2},$$ 
meaning that $\underline{j}(12_n749) = 3$ and $\ol{j}(12_n749) = 21$.  Since $21 \nleq 4(3) + 1 + 4 = 17$, Theorem \ref{generalizes to khovanov, cases by 2nd coeff} tells us that $12_n749$ is not positive. 
    
\end{proof}

We note that $12_n749$ is, however, a quasipositive knot. 

\begin{question}\label{inf pass fail}
    Can we construct an infinite family of non-positive quasipositive links with vanishing second Jones coefficient whose non-positivity is detected by the Khovanov homology test of Theorem \ref{generalizes to khovanov, cases by 2nd coeff} but is not detected by the Jones polynomial test of Theorem \ref{results for Jones}, as in the example of knot $12_n749$?
\end{question}

\begin{question} \label{exist pass fail for pm 1 or 2}    
    Do there exist non-positive quasipositive links with second Jones coefficient equal to $\pm 1$ or $\pm 2$ whose non-positivity is detected by the Khovanov homology test of Theorem \ref{generalizes to khovanov, cases by 2nd coeff} but is not detected by the Jones polynomial test of Theorem \ref{results for Jones}, similar to the example of knot $12_n749$? 
\end{question}

We expect Questions  \ref{inf pass fail} and \ref{exist pass fail for pm 1 or 2} will be answered in the affirmative. 

Finally, we note that $12_n749$ is an almost-strongly quasipositive knot but it is not strongly quasipositive. 

\begin{question}\label{pass fail for strong}
    Do there exist non-positive strongly quasipositive links which pass the test of the Jones polynomial but fail the test of Khovanov homology, similar to the example of knot $12_n749$? Or are both tests equally successful in detecting non-positivity among strongly quasipositive links?
\end{question}
\printbibliography

\end{document}